\documentclass[a4paper,20pt]{article}

\usepackage{amsmath, amsthm, amscd, amsfonts, graphicx,makeidx,amssymb,mathrsfs, float,capt-of}
\usepackage[all]{xy}
\usepackage{verbatim}
\usepackage[mathcal]{euscript}
\usepackage{graphicx,epsfig}
\usepackage{pifont}
\usepackage{color}
\definecolor{rojo}{rgb}{1,0,0} % valores de las componentes roja, verde y azul (RGB)

\newtheorem{thm}{Theorem}[section]
\newtheorem{cor}[thm]{Corollary}

\newtheorem{prop}[thm]{Proposition}

\theoremstyle{definition}

\theoremstyle{definition}

\theoremstyle{remark}

\renewcommand {\k} {\Bbbk}

\newcommand {\C} {\mathbb{C}}
\newcommand{\ben}{\begin{equation}}
\newcommand{\een}{\end{equation}}
\newcommand{\bena}{\begin{equation*}}
\newcommand{\eena}{\end{equation*}}

\newcommand{\ZZ}{{\mathcal Z}}

\newcommand{\RR}{\mathcal{R}}

\newcommand{\To}{\longrightarrow}
\newcommand{\Tot}{\longmapsto}

\newcommand{\op}{\operatorname}

\def\RR{\mathbb{R}}
\def\ZZ{\mathbb{Z}}

\newcommand{\al}{\alpha}

\newcommand{\de}{\Delta}

\newcommand{\rt}{\rightarrow}

\newcommand{\mb}{\mathbb}

\newdir{ >}{{}*!/-10pt/@{>}}

\newcommand{\var}{\varepsilon}

\title{Equivariant Topological Quantum Field Theory in Dimension 2}
\author{Ana Gonz\'alez\thanks{IMERL, Facultad de Ingenier\'ia, Montevideo, Uruguay} and Carlos Segovia\thanks{Mathematisches Institut, Universit\"at Heidelberg, Germany}\\}
\begin{document}
\maketitle

\begin{quote}\small\small {ABSTRACT. For $G$ a finite group,
we prove in dimension 2 that there is a monoidal equivalence between
the category of $G$-equivariant topological quantum field theories
and the category of $G$-Frobenius algebras, this was proved in
\cite{a10}. This work consists to give, in more detail, a proof of
this result. }
\end{quote}

\section{Introduction}
For $G$ a finite group, the $G$-cobordism category was introduced by
Turaev in \cite{turaev}. A ``\emph{linearization}" of this category
is given by a symmetric monoidal functor to the category of finite
dimensional vector spaces. It so remarkable that, in dimension 2,
this structure is defined completely by the algebraic object given
by a $G$-Frobenius algebra. This result was proved by Moore and
Segal in \cite{a10}, and more precisely this is expressed as
follows.

\begin{thm}\label{main}
In dimension 2, there is a monoidal equivalence between the category
of $G$-equivariant topological quantum field theories and the
category of $G$-Frobenius algebras.
\end{thm}
We finish this section with the definition of the principal two
objects of this work. They are the $G$-cobordism category in
dimension two and the $G$-Frobenius algebras.

The $G$-cobordism category in dimension $n$ has as objects
$(n-1)$-dimensional manifolds equipped with  principal $G$-bundles
and the morphisms are cobordisms with principal $G$-bundles over
them. We recall that a \emph{cobordism} from $\Sigma_0$ to
$\Sigma_1$ is an oriented $n$-manifold $M$ together with maps \bena
\Sigma_0\longrightarrow M \longleftarrow \Sigma_1\eena such that
$\Sigma_0$ maps diffeomorphically onto the in-boundary of $M$, and
$\Sigma_1$ maps diffeomorphically onto the out-boundary of $M$. Two
cobordisms from $\Sigma_0$ to $\Sigma_1$ are called
\emph{equivalent} if there is a diffeomorphism which commute the
diagram \bena\label{sense} \xymatrix{ &M'& \\ \Sigma_0\ar[ru]\ar[rd]
&&\Sigma_1\ar[ld]\ar[lu]\\ & M\ar[uu]^\cong & } \eena Similarly, we
can consider an identification in the cobordisms with principal
$G$-bundles considering $G$-equivariant diffeomorphisms.

A $G$-Frobenius algebra is a graded algebra $A=\oplus_{g\in G}A_g$,
where $A_g$ is a vector space of finite dimension, for all $g\in G$
such that
\begin{enumerate}
\item\label{axiom1} There is a homomorphism $\alpha:G\rt\op{Aut}(A)$,
with $\al_h:A_g\rt A_{hgh^{-1}}$, where $\op{Aut}(A)$ is the algebra of automorphisms of $A$. Moreover, for every $g\in G$ we have
\bena\al_g|_{A_g}=1_{A_g}\,.\eena
Note that $\al_e:A_g\rt A_g$ is the identity map.
\item\label{axiom2} There is a $G$-invariant trace $\var:A_e\rt \C$,
which induce non-degenerate pairings \bena\theta_g:A_g\otimes
A_{g^{-1}}\rt\C\,,\;\mbox{for all}\; g\in G.\eena
\item\label{axiom3} For all $\phi\in A_g$ and $\phi'\in A_h$ we have that the product is twisted commutative, i.e.
\bena \phi\phi'=\al_g\bigl(\phi'\bigr)\phi\,.\eena
\item\label{axiom4} Let $\Delta_g=\sum_i\xi_i^g\otimes\xi_i^{g^{-1}}\in A_g\otimes A_{g^{-1}}$ be the Euler element,
where $\left\{\xi_i^g\right\}$ is a base
of $A_g$ and $\left\{\xi_i^{g^{-1}}\right\}$ is the dual base in $A_{g^{-1}}$. We have, for all $g,h\in G$, that
\bena\sum_i\al_h\bigl(\xi_i^g\bigr)\xi_i^{g^{-1}}=\sum_i\xi_i^h\al_g\left(\xi_i^{h^{-1}}\right)\,.\eena
\end{enumerate}

\section{Algebraic data}
This section is devoted mainly to prove two facts about
$G$-Frobenius algebras. The first is Theorem \ref{teoabramsequiv}
which gives an equivalent way of see when a Frobenius trace is
non-degenerate. This result, for Frobenius algebras, originally
appears in \cite{lawvere} and it is proved in \cite{abar}. It is
presented in the general case in \cite{ana}. The second consists in
showing that the $G$-invariant part of a $G$-Frobenius algebra is in
fact a Frobenius algebra (see \cite{ana}).

\begin{thm}
\label{teoabramsequiv} Let $A=\oplus_{g\in G} A_g$ be a graded
algebra with twisted commutative products $m_{g,h}:A_g\otimes A_h\rt
A_{gh}$, unit $u:\C\rt A_e$ and an action $\alpha:G\To\op{Aut}(A)$,
where $A_g$ is a finite dimensional vector space, for all $g\in G$.
Then, there is a Frobenius trace $\varepsilon:A_e\To\C$ if and only
if there are  twisted cocommutative coproducts $\de_{g,h} :A_{gh}\rt
A_{g}\otimes A_{h}$, with trace $\varepsilon$, such that for every
$g,h,k\in G$ we have the commutativity of the following diagram
\ben\label{abram}\xymatrix@C=1.5cm@R1cm{ A_g\otimes A_{hk} \ar[r]^{m_{g,hk}}\ar[d]_{1\otimes\de_{h,k}} & A_{ghk}\ar[d]^{\de_{gh,k}}\\
A_g\otimes A_h\otimes A_k\ar[r]_{m_{g,h}\otimes 1}&A_{gh}\otimes A_k } \,.\een
\end{thm}
\begin{proof}  We just give the essential changes added to the proof given in \cite{abar}. The non-trivial part is the necessity, where
we take as coproducts the following,
\ben
\xymatrix@C=2cm@R1cm{
A_{gh}\ar[rr]^{\Delta_{g,h}}\ar[d]_{\Phi_{gh}}&&A_g\otimes A_h\\
A_{h^{-1}g^{-1}}^*\ar[r]_{{m_{h^{-1},g^{-1}}}^*}&A_{h^{-1}}^*\otimes
A_{g^{-1}}^*\ar[r]_{\Phi_h^{-1}\otimes\Phi_g^{-1}}&A_{h}\otimes
A_{g}\ar[u]_{\tau}
}\een
where $\Phi_g:A_g\To A_{g^{-1}}^*$ is the isomorphism defined by the
non-degenerate trace $\varepsilon:A_e\To\C$. The twisted
cocommutativity of the coproduct is a consequence of the following
two commutative diagrams.
As the product is twisted commutative, then we have that the next diagram commute
\bena
\xymatrix@C=2cm{
A_h\otimes A_g\ar[r]^{\al_g\otimes 1}\ar[d]_{\Phi_h\otimes \Phi_g}& A_{ghg^{-1}}\otimes A_g\ar[d]^{\Phi_{ghg^{-1}}\otimes \Phi_g}\\
A_{h^{-1}}^*\otimes A_{g^{-1}}^*\ar[r]_{\al_g^*\otimes 1}& A_{gh^{-1}g^{-1}}^*\otimes A_{g^{-1}}^*\,,
}
\eena
We deduce that
\bena
\xymatrix@C=2cm{
A_{gh}\ar[d]_{\Delta_{g,h}}\ar[r]^{\Delta_{ghg^{-1},g}}& A_{ghg^{-1}}\otimes A_g\\
A_g\otimes A_h\ar[r]_{\tau}&A_h\otimes A_g\ar[u]_{\al_g\otimes 1}
}
\eena commutes, then the coproduct is twisted cocommutative.
%\bena
%\xymatrix@C=1.5cm{A_{gh}\ar[r]^{\Phi_{gh}}\ar[d]_{\Phi_{gh}}&A_{h^{-1}g^{-1}}^*\ar[rr]^{m_{g^{-1},gh^{-1}g^{-1}}^*}&&A_{g^{-1}}^*\otimes A_{gh^{-1}g^{-1}}^*\ar[r]^{1\otimes \alpha_g^*}& A_{g^{-1}}^*\otimes A_{h^{-1}}^*\ar[dd]^{\Phi_g^{-1}\otimes\Phi_{h^{-1}}}\\
%A_{h^{-1}g^{-1}}^*\ar[d]_{m_{h^{-1},g^{-1}}^*}\ar[ru]_{1}&&&&\\
%A_{h^{-1}}^*\otimes A_{g^{-1}}^*\ar[rrr]_{\Phi_h^{-1}\otimes\Phi_g^{-1}}\ar[rrrruu]_\tau& &&A_h\otimes A_g\ar[r]_\tau&A_g\otimes A_h}
%\eena
%and
%\bena
%\xymatrix{
%A_{ghg^{-1}}\otimes A_g\ar[rr]^{\alpha_{g^{-1}}\otimes 1}\ar[d]_{\Phi_{ghg^{-1}}\otimes \Phi_g} && A_h\otimes A_g \ar[d]^{\Phi_h\otimes \Phi_g} \\
%A_{gh^{-1}g^{-1}}^*\otimes A_{g^{-1}}^*\ar[rr]_{\alpha_g^*\otimes 1}&& A_{h^{-1}}^*\otimes A_{g^{-1}}^*}\,.
%\eena

The coassociativity of the coproduct is an immediately  consequence of the associativity of the product.

%\bena\scalebox{.9}{\xymatrix{A_{ghk}\ar[r]^{\Phi_{ghk}}\ar[d]^{\Phi_{ghk}}& A_{k^{-1}h^{-1}g^{-1}}^*\ar[r]^{m_{k^{-1}h^{-1},g^{-1}}^*} & A_{k^{-1}h^{-1}}^*\otimes A_{g^{-1}}^*\ar[r]^\tau\ar[ddd]^{m_{k^{-1},h^{-1}}^*\otimes 1}& A_{g^{-1}}^*\otimes A_{k^{-1}h^{-1}}^*\ar[r]^{\Phi_g^{-1}\otimes\Phi_{hk}^{-1}}\ar[rd]_{\Phi_g^{-1}\otimes 1} & A_g\otimes A_{hk}\ar[d]_{1\otimes \Phi_{hk}} \\
%A_{k^{-1}h^{-1}g^{-1}}^*\ar[ru]_{1}\ar[dd]^{m_{k^{-1},h^{-1}g^{-1}}^*}&&&& A_g\otimes A_{k^{-1}h^{-1}}^*\ar[d]_{1\otimes m_{k^{-1},h^{-1}}^*}\\
%&&&&A_g\otimes A_{k^{-1}}^*\otimes A_{h^{-1}}^*\ar[d]_{1\otimes\tau}\\
%A_{k^{-1}}^*\otimes A_{h^{-1}g^{-1}}^*\ar[rr]_{1\otimes m_{h^{-1},g^{-1}}^*}\ar[d]^\tau&& A_{k^{-1}}^*\otimes A_{h^{-1}}^*\otimes A_{g^{-1}}^*\ar[d]_{(13)}\ar[r]^{1\otimes 1\otimes \Phi_g^{-1}}& A_{k^{-1}}^*\otimes A_{h^{-1}}^*\otimes A_g\ar[r]^{(13)} &A_g\otimes A_{h^{-1}}^*\otimes A_{k^{-1}}^*\ar[dd]_{1\otimes \Phi_h^{-1}\otimes\Phi_k^{-1}}\\
%A_{h^{-1}g^{-1}}^*\otimes A_{k^{-1}}^*\ar[d]^{\Phi_{gh}^{-1}\otimes\Phi_k^{-1}}\ar[rd]^{1\otimes \Phi_k^{-1}}&& A_{g^{-1}}^*\otimes A_{h^{-1}}^*\otimes A_{k^{-1}}^*\ar[rd]_{1\otimes 1\otimes \Phi_k^{-1}}&&\\
%A_{gh}\otimes A_k\ar[r]_{\Phi_{gh}\otimes 1}& A_{h^{-1}g^{-1}}^*\otimes A_k\ar[r]^{m_{h^{-1},g^{-1}}^*\otimes 1}& A_{h^{-1}}^*\otimes A_{g^{-1}}^*\otimes A_k\ar[r]_{\tau\otimes 1}& A_{g^{-1}}^*\otimes A_{h^{-1}}^*\otimes A_k\ar[r]^{\Phi_g^{-1}\otimes\Phi_h^{-1}\otimes 1} &A_g\otimes A_h\otimes A_k}
%}\eena
%where we use that the product is associative and that the maps $\Phi_g$ have an $A$-module property.
Now, to prove the commutativity of diagram \eqref{abram} we define the map
\bena
\xymatrix@C=2cm{A_{gh}\ar[rrr]^{\overline{m}_{g,h^{-1}}}\ar[d]_{\Phi_{gh}}&&&A_g\otimes A_{h^{-1}}^*\\
A_{h^{-1}g^{-1}}^*\ar[rr]_{m_{h^{-1},g^{-1}}^*}&&
A_{h^{-1}}^*\otimes A_{g^{-1}}^*\ar[r]_{1\otimes\Phi_g^{-1}}
&A_{h^{-1}}^*\otimes A_{g}\ar[u]_{\tau}}\,, \eena and the diagram
\eqref{abram} is inside the following \ben\label{otro} \xymatrix{&&
A_g\otimes A_{hk}\ar[lld]_{1\otimes
\overline{m}_{h,k^{-1}}}\ar[rr]^{m_{g,hk}}\ar[d]^{1\otimes
\Delta_{h,k}}&&A_{ghk}\ar[rrd]^{\overline{m}_{gh,k^{-1}}}\ar[d]^{\Delta_{gh,k}}&&\\A_g\otimes
A_h\otimes A_{k^{-1}}^*\ar[rr]_{1\otimes 1\otimes
\Phi_k^{-1}}\ar[rrd]_{m_{g,h}\otimes 1}&&A_g\otimes A_h\otimes
A_k\ar[rr]_{m_{g,h}\otimes 1}&&A_{gh}\otimes A_k\ar[d]_{1\otimes
\Phi_k}&&A_{gh}\otimes
A_{k^{-1}}^*\ar[ll]^{1\otimes\Phi_k^{-1}}\ar[lld]^{1\otimes 1}\\&&
A_{gh}\otimes A_{k^{-1}}^*\ar[rr]_{1\otimes 1}&& A_{gh}\otimes
A_{k^{-1}}^*&&} \een The commutativity of the other diagrams is a
natural consequence of the definition of $\overline{m}_{h,k^{-1}}$,
so \eqref{abram} must commute.
\end{proof}

The preview theorem implies a property which occur in the coproducts $\Delta_{g,h}$. This is a compatibility between the left and right module structure of a $G$-Frobenius algebra.

\begin{cor}
For a basis $\bigl\{\xi_i^k\bigr\}$ of $A_k$ and $\bigl\{\xi_i^{k^{-1}}\bigr\}$ its dual basis in $A_{k^{-1}}$, we have
\begin{equation}\label{coproducto}\Delta_{g,h}(\phi)=\sum\phi \xi_i^{h^{-1}}\otimes\xi_i^h=\sum
\xi_i^g\otimes\xi_i^{g^{-1}}\phi\,.
\end{equation}
\end{cor}
\begin{proof}
Similarly as before, we can define the map $\overline{\overline{m}}_{g^{-1},h}$ as
\bena
\xymatrix@C=2cm{A_{gh}\ar[rr]^{\overline{\overline{m}}_{g^{-1},h}}\ar[d]_{\Phi_{gh}}& &A_{g^{-1}}^*\otimes A_h\\ A_{h^{-1}g^{-1}}\ar[r]_{m_{h^{-1},g^{-1}}^*} & A_{h^{-1}}^*\otimes A_{g^{-1}}^*\ar[r]_{\Phi_h^{-1}\otimes 1}& A_{h}\otimes A_{g^{-1}}^*\ar[u]_{\tau}}\,.
\eena
We can make a diagram analogous to \eqref{otro}. Since the coproduct is written in two ways $$\Delta_{g,h}=\bigl(1\otimes\Phi_h^{-1}\bigr)\overline{m}_{g,h^{-1}}= \bigl(\Phi_g^{-1}\otimes 1\bigr)\overline{\overline{m}}_{g^{-1},h},$$ we have the identity \eqref{coproducto}.
\end{proof}

Finally, we end this section with the proof that the stuff of this article generalizes the usual theory of Frobenius algebras. This result appear in \cite{ana}. There are some concrete examples of this structure provided by some cohomology theories associated to orbifolds, see \cite{art,chruan} and by $G$-string topology, see \cite{luxion1,art}.

\begin{prop}\label{orb}
If $A$ is a $G$-Frobenius algebra we have that the $G$-invariant part of this algebra, denoted by $A_{orb}$, is a Frobenius algebra.
\end{prop}
\begin{proof}
 Set $A_{orb}:=A^G=\left(\oplus_{g\in G}A_g\right)^G$. Note that $\displaystyle{A_{orb}\cong\oplus_{g\in T}A_g^{C(g)}}$ where $T$ is a set of representatives for the conjugacy classes in $G$ and $C(g)$ is the centralizer of $g\in G$. The maps are
$$\Psi:\bigoplus_{g\in T}A_g^{C(g)}\rt\left(\bigoplus_{b\in G}A_g\right)^G\quad\mbox{defined by}\quad \sum_{g\in G}\phi'_g\mapsto\sum_{g\in T}\sum_{h\in[g], h=kgk^{-1}}\al_k(\phi'_g)$$
 and
$$\Upsilon:\left(\bigoplus_{b\in G} A_g\right)^G\rt \bigoplus_{g\in T} A_g^{C(g)}\quad\mbox{defined by}\quad \sum_{g\in G}\phi_g\mapsto\sum_{g\in T}\phi_g.$$
First, we prove that $A_{orb}$ is an algebra. The product is  the restriction of the product in $A$ because, for $\phi,\phi'\in A_{orb}$, we have that $g\cdot \phi=\al_g(\phi)=\phi$ and $g\cdot \phi'=\al_g(\phi')=\phi'$, for all $g\in G$. Then $g\cdot\phi\phi'=\al_g(\phi\phi')=\al_g(\phi)\al_g(\phi')=\phi\phi'$. An additional property is the commutativity of the product, for this we take $\phi=\sum_{g\in G}\phi_g$ and $\phi'=\sum_{h\in G}\phi'_h\in A_{orb}$, then
$$\phi\phi'=\sum_{g\in G}\sum_{h\in G}\phi_g\phi'_h=\sum_{g,h\in G}\al_g(\phi'_h)\phi_g
=\sum_{g\in G}\al_g\left(\sum_{h\in G}\phi'_h\right)\phi_g=\sum_{g\in G}\phi'\phi_g=\phi'\phi\,.$$\\
For the Frobenius structure we define the trace $\var:A_{orb}\rt \mb{C}$ as the restriction of the trace $\var:A\rt \mb{C}$, which is zero on $A_g$ when $g\not =e$.\\
 To complete the prove we need to see that the induced pairing is nondegenerate. We take $\phi=\sum_{g\in G}\phi_g\in A_{orb}$ and we suppose $\var(\phi\phi')=0$,
for all $\phi'\in A_{orb}$. We well prove that $\phi=0$. If we prove that $\phi_g=0$ for all $g\in T$ we finish, because $\phi=\sum_{g\in
T}\sum_{h\in [g], h=kgk^{-1}}\al_k(\phi_g)$. We can consider $\phi'_h\in A_h$, where $h$ is the representant of $[h]\in T$, then
$\phi'=\sum_{k\in [h], k=lhl^{-1}}\al_l(\phi'_h)\in A_{orb}$.
Now
$$\var(\phi\phi')=|[h]|\var(\phi_{h^{-1}}(\phi'_h)).$$
 Hence $\var(\phi_{h^{-1}}\phi'_h)=0$, for all $\phi'_h\in A_h$. Then $\phi_{h^{-1}}=0$ for every $h\in T$, therefore $\phi=0$.
\end{proof}

\section{Proof of the main theorem}
\label{section2}
The non-trivial part of theorem \ref{main} consists that the algebraic data given by a $G$-Frobenius algebra is enough to cut every $G$-cobordism in all possible way. For this we need a description of the elementary components of the $G$-cobordism category in dimension 2. For the objects we take a fixed circle, and every principal $G$-bundle is described by taking based points in the base and the total space and a based projection. Thus, they are exactly the elements of the group, where the bijection is given by the lifting of the base space starting in the base point of the total space. For $g\in G$ we denote by $P_g$ the total space of the principal $G$-bundle associated. For the morphisms we do not have an explicit description, but we can described the elemental components:
\begin{enumerate}
\item for $g,h\in G$, the morphisms from $g$ to $h$ (with base space the cylinder) are in one-to-one correspondence with the elements of the set $\{k:h=kgk^{-1}\}$ up to the identification\footnote{This identification is given by the action of the Dehn twists.} $k\sim h^nk g^m$,
where $n,m\in\ZZ$. A typical element is
\begin{equation*}\label{cyl}\includegraphics{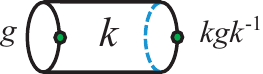}\,.\end{equation*}
 This correspondence is given by means of the homotopy lifting property applied to the base space, a cylinder, with starting point $P_g$;

\item since the pair of pants has the homotopy type of $S^1\vee S^1$, we can describe the principal $G$-bundles over the circle as follows. We start with the principal $G$-bundles over the \emph{thin} pair of pants $S^1\vee S^1$. They are in bijection with the group homomorphisms from the fundamental group of $S^1\vee S^1$ to $G$, thus with $G\times G$. A basic element is
\begin{equation*}\label{pants}\includegraphics{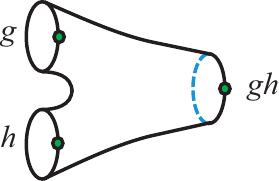}\,.\end{equation*}
Any other principal $G$-bundle over the pair of pants is obtained by composition with principal $G$-bundles of the cylinder; and

\item the disk which is contractile, therefore it has only one principal $G$-bundle over it, which is trivial.
\end{enumerate}
It is important to mention that there are some invertible cobordisms which we put them away. They are the invertible cobordisms resulting by a cylinder construction associated to every diffeomorphism of the circle. We can dismiss these elements since, we can consider the subcategory of topological field theories which send them to the identity maps.

Again it is clearly that every $G$-equivariant topological field theory defines the structure of a $G$-Frobenius algebra. But conversely, we start with the assignation of a linear application to every elemental component as in the following table,
\bena
\begin{tabular}{|c|c|}
  \hline
  \hline
   Component & Linear object \\
   \hline\hline
  \includegraphics[scale=0.7]{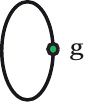} &  $A_g$\\\hline
  \includegraphics[scale=0.7]{f10.pdf}&  $m_{g,h}:A_g\otimes A_h\To A_{gh}$\\\hline
   \includegraphics[scale=0.7]{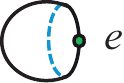}&  $u:\k\To A_e$\\\hline
   \includegraphics[scale=0.7]{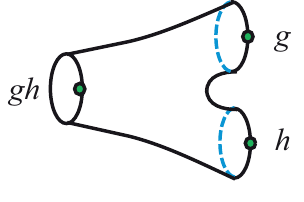}&  $\Delta_{g,h}:A_{gh}\To A_g\otimes A_h$\\\hline
   \includegraphics[scale=0.7]{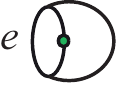}&  $\varepsilon:A_e\To\k$\\\hline
   \includegraphics[scale=0.7]{f11.pdf}&  $\alpha_k:A_g\To A_{kgk^{-1}}$\\
   \hline
  \hline
\end{tabular}
\eena
with $e\in G$ the identity element.
Since a cobordisms could be represented by different decomposition in elemental components, then every decomposition has associated a linear map. We have to prove that all these linear maps are the same. We should first check if the association of the linear maps of the last table is well defined. The following two cases are of relevance:
\begin{enumerate}
\item the first consists on the invariance under the difeomorphisms over the cylinder. We know, by the theory of mapping class group, that they are generated by the Dehn twist given by $(e^{i\theta},t)\Tot (e^{i(\theta+t2\pi)},t)$. This is exactly that the action $\alpha_g$ is trivial when we restrict to the component $A_g$. We exemplified this in the following picture,
    \bena
    \includegraphics{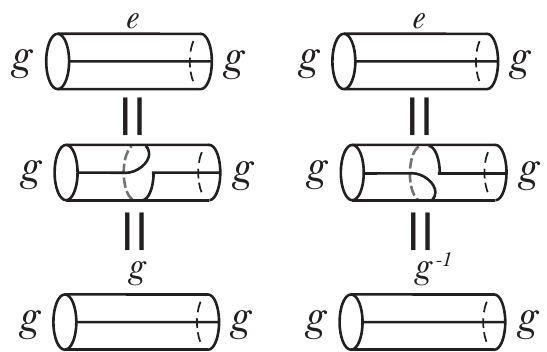}\,;
    \eena
\item for the pair of pants the maps $m_{g,h}:A_g\otimes A_h\To A_{gh}$ depends of the base point that we take on the outgoing boundary  circle,
$$\includegraphics{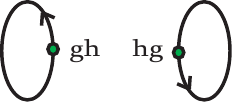}\,;$$
each of the cases are just choosing an ordering of the incoming boundaries and are related by the conjugation
$$\alpha_{h}:A_{gh}\To A_{hg},$$
so the consistency for us to have a well-defined assignment is that
\begin{equation}\label{multip}
m_{h,g}(\psi_2\otimes\psi_1)=\alpha_h\bigl(m_{g,h}(\psi_1\otimes\psi_2)\bigr)\,.
\end{equation}
This is just the calculation
\begin{equation*}
\begin{split}
\alpha_h\bigl(m_{g,h}(\psi_1\otimes\psi_2)\bigr)&=m_{hgh^{-1},h}\bigl(\alpha_h(\psi_1)\otimes\alpha_h(\psi_2)\bigr)\\
&=m_{hgh^{-1},h}\bigl(\alpha_h(\psi_1)\otimes(\psi_2)\bigr)=m_{h,g}(\psi_2\otimes\psi_1):=\psi_2\psi_1,
\end{split}
\end{equation*}
where we use that $G$ acts by algebra automorphisms, the twisted commutativity of the product and the property of the item 1.
\end{enumerate}

Therefore, the linear maps associated to the elemental components are well defined. We pass through the proof that for every cobordism and every of its decomposition in elemental components their associated linear maps coincide. Let $M$ be a cobordism, the set of all its decompositions can be modeled by the space of real smooth functions $f:M\To\RR$. The study of this space is called \emph{Cerf theory}, see \cite{a11}. Every decomposition can be modeled by an \emph{excellent} function, i.e. a function where all the critical points are of Morse type and all the critical values are distinct. For a cobordism $M$, we take an excellent function with critical points $x_1,...,x_k\in M$ and critical values $c_1,...,c_k\in M$ with $c_j=f(x_j)$. For the sequence of values $t_0,t_1,...,t_k$ with
\bena
0=t_0<c_1<t_1<c_2<t_2<...<c_k<t_k=1\,,
\eena
the pre-image
 $S_t=f^{-1}(t)$ is a disjoint union of circles, and $M_i=f^{-1}([t_{i-1},t_i])$ is a cobordism from $S_{t_{i-1}}$ to $S_{t_i}$. This cobordism is a disjoint union of cylinder together with one pair of pants or one disc. Every decomposition of $M$ in elemental components is given by an excellent function. We can take a path that connect every two excellent functions. Moreover, every point of this path is an excellent function except for a finite set of points given by two cases:
\begin{enumerate}
\item\label{d1} we have one point of birth, i.e. locally of the form
    \bena -x_1^2-\cdots -x_i^2+x_{i+1}^2+\cdots x_{n-1}^2+x_n^2;\eena
\item\label{d2} two critical points has the same critical value and all are of Morse type. All of them are distinct except this two.
\end{enumerate}
Thus for the proof of theorem it remain to solve these two cases.

For the bird points, of the form \eqref{d1}, the only possibilities are the corresponding to the
unit and counit (or trace) of $A$. We depict them in the following
pictures,
$$\includegraphics{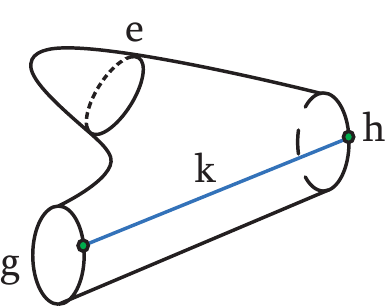}\hspace{2cm}\includegraphics{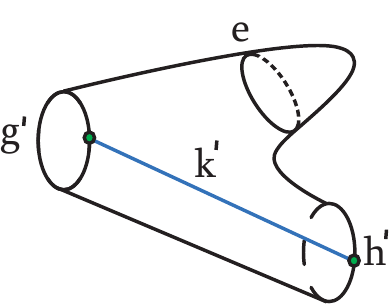}$$
with $e\in G$ the identity, $h=kgk^{-1}$ and $h^\prime=k^\prime
g^\prime {k^\prime}^{-1}$.

For elements of the form \eqref{d2}, let $C_i$ be the number of critical points of index $i$ of a Morse function. Morse theory provide the following relation with the Euler characteristic
 \bena
 \chi=\sum_i(-1)^iC_i\,.
 \eena
 In dimension 2, the cobordism associated to the two critical points, has Euler characteristic inside $\{-2,0,2\}$. For $\chi=2$ and $\chi=0$ we have the cases of the sphere and the cylinder respectively. We denote by $(m,g,n)$, the 2 dimensional cobordism with $m$ entries, $g$ genus and $n$ exits. Thus the case $\chi=-2$ corresponds to the cases $(1,1,1)$, $(1,0,3)$, $(3,0,1)$ and $(2,0,2)$. These are consider in detail below.

For $(1,1,1)$ we associate holonomies to specific arcs that connect the two critical points. For example we do the following assignation,
\bena\includegraphics{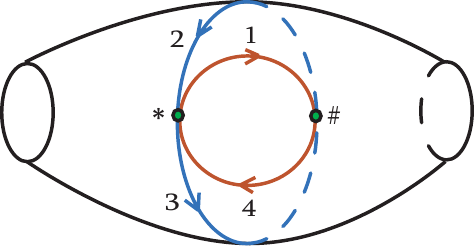}\eena
where the numbers 1,2,3, and 4 stand for elements in $G$. There are two possible cases, these are to identify firstly the critical point $*$ and then $\#$. Thus, there are two possible decompositions, which are presented in the following pictures.
\bena\includegraphics{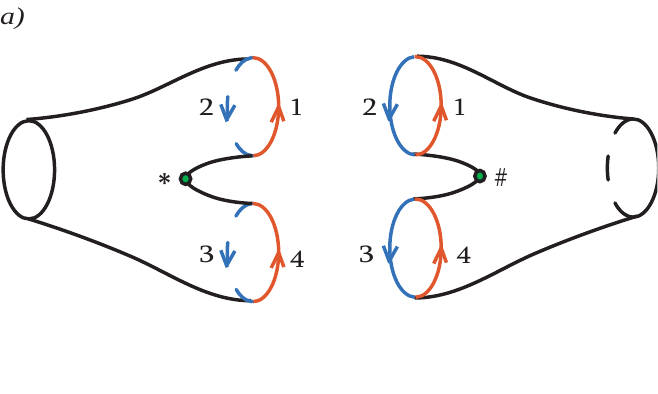}\hspace{1cm}\includegraphics{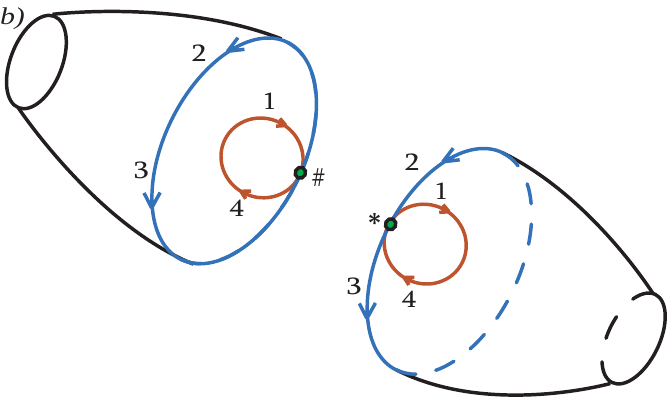}\eena
These two decomposition are related by the following diffeomorphism,
\bena\includegraphics{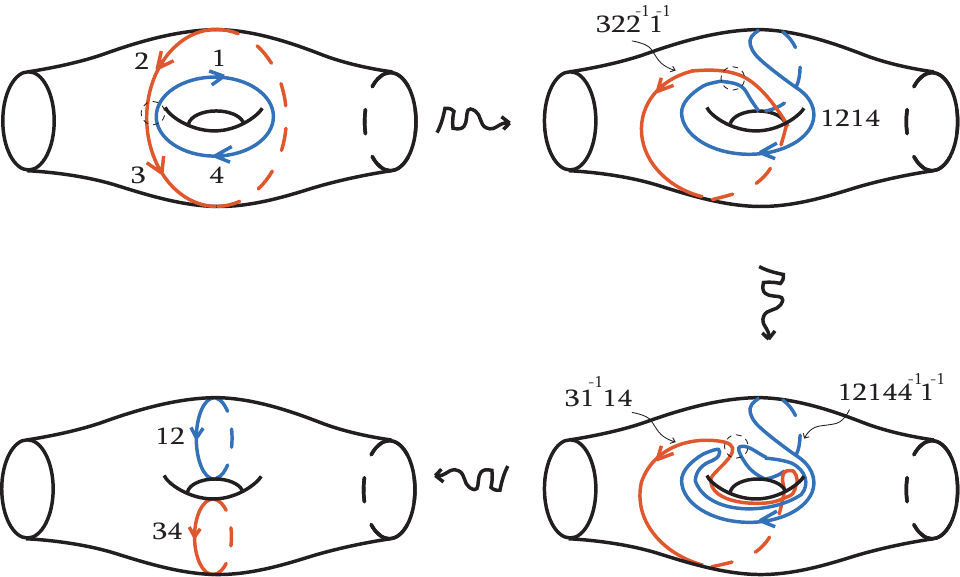}\eena
We will finish this case if we prove the independence between the
following compositions.

\begin{enumerate}
\item[a)]$ A_{1234}\To
A_{12}\otimes A_{34}\simeq A_{21}\otimes A_{43}\To A_{2143}\simeq A_{1432}$
\item[b)]
$ A_{1234}\simeq A_{4123}\To A_{41}\otimes A_{23}\simeq A_{14}\otimes A_{32}\To A_{1432}$
\end{enumerate}
For a) we have the assignation
\begin{equation*}
\begin{split}
\phi\Tot\sum\phi\xi_i^{4^{-1}3^{-1}}\otimes\xi_i^{34}\stackrel{\sim}{\Tot}\sum\alpha_2\bigl(\phi\xi_i^{4^{-1}3^{-1}}\bigr)\otimes\alpha_4\bigl(\xi_i^{34}\bigr)\Tot&\sum\alpha_2
\bigl(\phi\xi_i^{4^{-1}3^{-1}}\bigr)\alpha_4\bigl(\xi_i^{34}\bigr)\stackrel{\sim}{\Tot}\\
\sum\alpha_{2^{-1}}\left(\alpha_2\bigl(\phi\xi_i^{4^{-1}3^{-1}}\bigr)\alpha_4(\xi_i^{34})\right)&=\sum \phi\xi_i^{4^{-1}3^{-1}}\alpha_{2^{-1}4}\bigl(\xi_i^{34}\bigr)
\end{split}
\end{equation*}
and b) is

\begin{alignat*}{2}
\phi\stackrel{\sim}{\Tot}\alpha_4(\phi)\Tot\sum\alpha_4(\phi)\xi_i^{3^{-1}2^{-1}}\otimes\xi_i^{23}&\stackrel{\sim}{\Tot}\sum\alpha_{4^{-1}}\left(\alpha_4(\phi)\xi_i^{3^{-1}2^{-1}}\right)
\otimes\alpha_3\bigl(\xi_i^{23}\bigr)\Tot\\
\sum\phi\alpha_{4^{-1}}\bigl(\xi_i^{3^{-1}2^{-1}}\bigr)\alpha_{3}\bigl(\xi^{23}\bigr)&=\sum\phi\xi_i^{4^{-1}3^{-1}2^{-1}4}\alpha_{34}\bigl(\xi^{4^{-1}234}\bigr)\\
&=\sum\phi\alpha_{4^{-1}3^{-1}2^{-1}4}\bigl(\xi_i^{34}\bigr)\xi^{4^{-1}3^{-1}}\\
&=\sum\phi\alpha_{4^{-1}3^{-1}}\left(\alpha_{2^{-1}4}\bigl(\xi_i^{34}\bigr)\xi^{4^{-1}3^{-1}}\right)\\
&=\sum\phi\xi_i^{4^{-1}3^{-1}}\alpha_{2^{-1}4}\bigl(\xi_i^{34}\bigr)
\end{alignat*}
where we use firstly the axiom \eqref{axiom4} of the definition of a $G$-Frobenius algebra and then the property \eqref{multip}. Note also that, as a consequence of the non-degeneration of the trace, we have that the Euler element is invariant over the action of $G$ and moreover, we can assume that $$\alpha_h\bigl(\xi_i^g\bigr)=\xi_i^{hgh^{-1}}$$ for basis
$\bigl\{\xi_i^g\bigr\}$ and $\bigl\{\xi_i^{hgh^{-1}}\bigr\}$ of $ A_g$ and $ A_{hgh^{-1}}$ respectively.

For the sphere with four holes, we start with the case $(2,0,2)$, i.e. two incoming circles, genus zero and two output circles. We assign the following holonomies,
$$\includegraphics{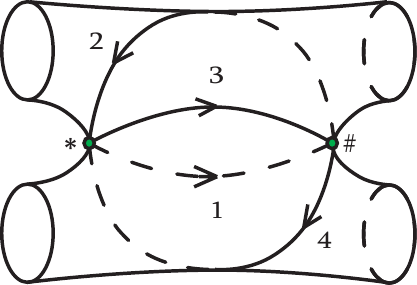}$$
We first consider the decomposition given by a vertical cut. This produce two possible maps according to the order of the critical points,
\begin{enumerate}
\item[a)]$ A_{12}\otimes A_{34}\To A_{1234}\simeq A_{2341}\To A_{23}\otimes A_{41}\simeq A_{32}\otimes A_{14}$
\item[b)]$ A_{12}\otimes A_{34}\simeq A_{21}\otimes A_{43}\To A_{2143}\simeq A_{3214}\To A_{32}\otimes A_{14}$
\end{enumerate}
For a)
\begin{alignat*}{2}
\phi\otimes\phi^\prime\Tot\phi\phi^\prime\stackrel{\sim}{\Tot}\alpha_{1^{-1}}(\phi\phi^\prime)\Tot\sum\alpha_{1^{-1}}(\phi\phi^\prime)\xi_i^{1^{-1}4^{-1}}\otimes\xi_i^{41}&\hspace{5cm}\\
\stackrel{\sim}{\Tot}\sum\alpha_{3}\left(\alpha_{1^{-1}}(\phi\phi^\prime)\xi_i^{1^{-1}4^{-1}}\right)\otimes\alpha_1\bigl(\xi_i^{41}\bigr)=
\sum\alpha_{3}&\left(\alpha_{1^{-1}}(\phi\phi^\prime)\xi_i^{1^{-1}4^{-1}}\right)\otimes\xi_i^{14}
\end{alignat*}
and b)
\begin{alignat*}{2}
\phi\otimes\phi^\prime\stackrel{\sim}{\Tot}\alpha_{1^{-1}}(\phi)\otimes\alpha_{3^{-1}}(\phi^\prime)\stackrel{\sim}{\Tot}\alpha_{1^{-1}}(\phi)\alpha_{3^{-1}}(\phi^\prime)&
\Tot\alpha_3\left(\alpha_{1^{-1}}(\phi)\alpha_{3^{-1}}(\phi^\prime)\right)=\alpha_{31^{-1}}(\phi)\phi^\prime\\
\Tot\sum\alpha_{31^{-1}}(\phi)\phi^\prime\xi_i^{4^{-1}1^{-1}}\otimes\xi_i^{14}&=\sum\alpha_{31^{-1}}(\phi)\alpha_1\left(\alpha_{-1}(\phi^\prime)\xi_i^{1^{-1}4^{-1}}\right)\otimes\xi_i^{14}\\
&=\sum\alpha_{31^{-1}}(\phi)\alpha_3\left(\alpha_{-1}(\phi^\prime)\xi_i^{1^{-1}4^{-1}}\right)\otimes\xi_i^{14}\\
&=\sum\alpha_{3}\left(\alpha_{1^{-1}}(\phi\phi^\prime)\xi_i^{1^{-1}4^{-1}}\right)\otimes\xi_i^{14}
\end{alignat*}
For the horizontal cut, which we depicted in the following two figures,
$$\includegraphics{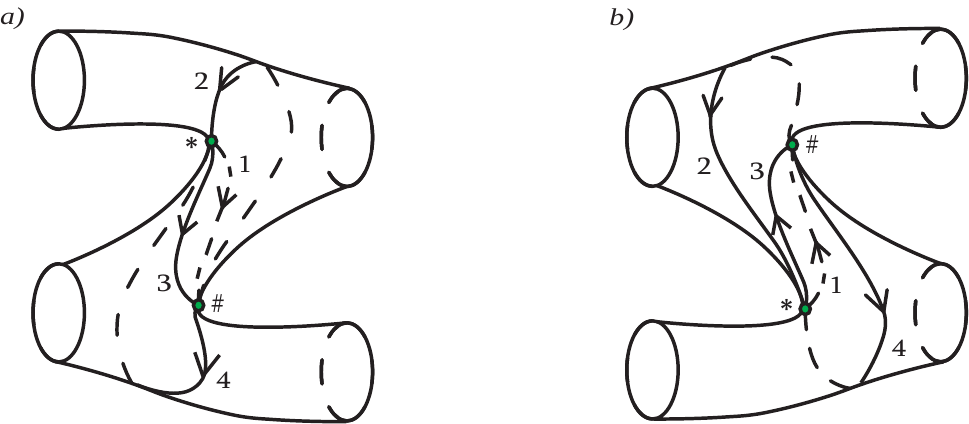}$$
This produces the following compositions
$$ A_{12}\otimes A_{34}\To A_{12}\otimes A_{31^{-1}}\otimes A_{14}\simeq A_{21}\otimes A_{1^{-1}3}\otimes A_{14}\To A_{23}\otimes A_{14}\simeq A_{32}\otimes A_{14}$$
$$\textrm{and}$$
$$ A_{12}\otimes A_{34}\simeq A_{21}\otimes A_{34}\To A_{23}\otimes A_{3^{-1}1}\otimes A_{34}\simeq A_{32}\otimes A_{13^{-1}}\otimes A_{34}\To A_{32}\otimes A_{14}$$
For the first, we have the calculation
\begin{alignat*}{2}
\phi\otimes\phi'\Tot \sum\phi\otimes\phi'\xi_i^{4^{-1}1^{-1}}\otimes \xi_i^{14}\stackrel{\sim}{\Tot}
\sum\alpha_{1^{-1}}(\phi)\otimes \alpha_{1^{-1}}\left( \phi'\xi_i^{4^{-1}1^{-1}}\right)\otimes \xi_i^{14}&\\
\Tot \sum\alpha_{1^{-1}}\left(\phi'\xi_i^{4^{-1}1^{-1}}\right)\otimes\xi_i^{14} \stackrel{\sim}{\Tot}\sum\alpha_3\left(\alpha_1^{-1}(\phi)\alpha_{1^{-1}}\bigl(\phi'\xi_i^{4^{-1}1^{-1}}\bigr)\right)\otimes \xi_i^{14}&\\
=\sum\alpha_3\left( \alpha_{1^{-1}}(\phi\phi')\xi_i^{1^{-1}4^{-1}}\right)\otimes \xi_i^{14}\,,&
\end{alignat*}
and for the second
\begin{alignat*}{2}
\phi\otimes\phi'\stackrel{\sim}{\Tot}\alpha_2(\phi)\otimes \phi'\Tot& \sum \alpha_2(\phi)\xi_i^{1^{-1}3}\otimes \xi_i^{3^{-1}1}\otimes \phi'=\sum \xi_i^{23}\otimes \xi_i^{3^{-1}2^{-1}}\alpha_2(\phi)\otimes \phi'\\
\stackrel{\sim}{\Tot}\sum\alpha_3\bigl(\xi_i^{23}\bigr)\otimes\alpha_3\left(\xi_i^{3^{-1}2^{-1}}\alpha_2(\phi)\right)\otimes \phi'&=\sum\xi_i^{32}\otimes \xi_i^{2^{-1}3^{-1}}\alpha_{32}(\phi)\otimes\phi'=\sum\alpha_{32}(\phi)\phi'\xi_i^{4^{-1}1^{-1}}\otimes\xi_i^{14}\,.
\end{alignat*}
So, it rest to write the following identities
\bena
\alpha_{31^{-1}}(\phi)\phi'\xi_i^{4^{-1}1^{-1}}=\alpha_{32}(\phi)\phi'\xi_i^{4^{-1}1^{-1}}=\alpha_{31^{-1}}\left(\phi\alpha_{13^{-1}}\bigl(\phi'\xi_i^{4^{-1}1^{-1}}\bigr)\right)
\eena
\bena
=\alpha_3\left(\alpha_{1^{-1}}\bigl(\phi\phi'\xi_i^{4^{-1}1^{-1}}\bigr)\right)=\alpha_3\left(\alpha_{1^{-1}}(\phi\phi')\xi_i^{1^{-1}4^{-1}}\right)\,,
\eena
where we use the identity \eqref{coproducto}.
%$$\includegraphics{fig10.pdf}$$

We shall do one more case, this is for $(3,0,1)$. The possibilities
are

$$\includegraphics{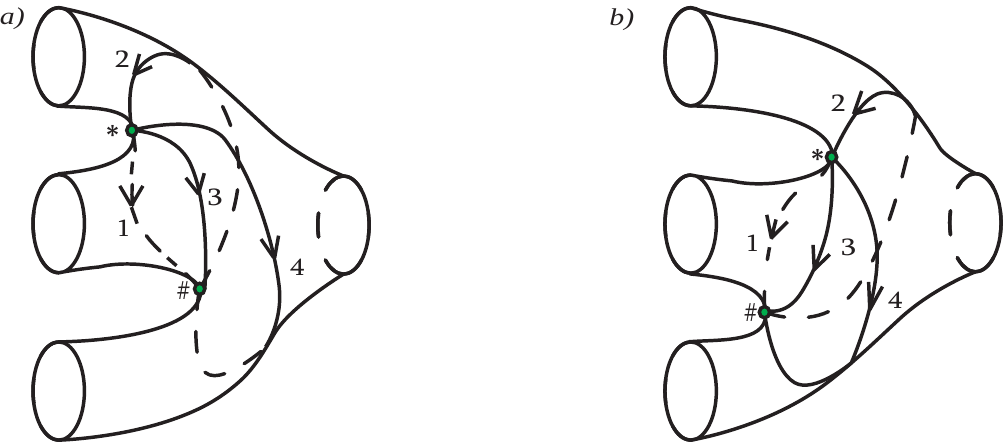}$$

with the maps
\begin{enumerate}
\item[a)]
$ A_{12}\otimes A_{31^{-1}}\otimes A_{43^{-1}}\To A_{1231^{-1}}\otimes A_{43^{-1}}\simeq A_{23}\otimes A_{3^{-1}4}\To A_{24}$\\
$ A_{12}\otimes A_{31^{-1}}\otimes A_{43^{-1}}\To A_{12}\otimes A_{31^{-1}43^{-1}}\simeq A_{21}\otimes A_{1^{-1}4}\To A_{24}$
\item[b)] $ A_{12}\otimes A_{31^{-1}}\otimes A_{43^{-1}}\simeq A_{21}\otimes A_{1^{-1}3}\otimes A_{3^{-1}4}\To A_{21}\otimes A_{1^{-1}4}\To A_{24}$\\
$ A_{12}\otimes A_{31^{-1}}\otimes A_{43^{-1}}\simeq A_{21}\otimes A_{1^{-1}3}\otimes A_{3^{-1}4}\To A_{23}\otimes A_{3^{-1}4}\To A_{24}$
\end{enumerate}
The case $(1,0,3)$ is the associated to the coproduct.

\bibliographystyle{amsalpha}
\bibliography{biblio1}
\end{document}